\documentclass[11pt]{amsart}

\usepackage{amsmath}
\usepackage{amssymb}
\usepackage{amsthm}
\usepackage{graphicx}
\usepackage{psfrag}
\usepackage{verbatim}

\usepackage[margin=1.2in]{geometry}
\def\E{\mathbb{E}}

\def\E{\mathbb{E}}
\def\R{\mathbb{R}}

\def\Z{\mathbb{Z}}

\def\eps{\epsilon}

\def\1{\mathbf{1}}

\def\lam {\lambda}
\def\ER{Erd\H{o}s-R\'{e}nyi }

\def\tor{\mathcal T^d}

\newtheorem*{theorem*}{Theorem}
\newtheorem{theorem}{Theorem}

\newtheorem{cor}{Corollary}
\newtheorem{defn}{Definition}
\newtheorem{prop}{Proposition}
\newtheorem*{prop*}{Proposition}
\newtheorem{conj}{Conjecture}

\begin{document}
\title{Birthday inequalities, repulsion,  and hard spheres}
\author{Will Perkins}
\address{School of Mathematics, University of Birmingham, UK.  E-mail: math@willperkins.org.}
\date{\today}

\begin{abstract}
We study a \textit{birthday inequality} in random geometric graphs: the probability of the empty graph is upper bounded by the product of the probabilities that each edge is absent.  We show the birthday inequality holds at low densities, but does not hold in general. We give three different   applications of the birthday inequality in statistical physics and combinatorics: we prove lower bounds on the free energy of the hard sphere model and upper bounds on the number of independent sets and matchings of a given size in $d$-regular graphs.

The birthday inequality is implied by a \textit{repulsion inequality}:  the expected volume of the union of spheres of radius $r$ around $n$ randomly placed centers increases if we condition on the event that the centers are at pairwise distance greater than $r$.  Surprisingly we show that the repulsion inequality is not true in general, and in particular that it fails in $24$-dimensional Euclidean space: conditioning on the pairwise repulsion of centers of $24$-dimensional spheres can \textit{decrease} the expected volume of their union.
\end{abstract}

\maketitle

\section{Introduction}

How many people must be in a room so that the chance at least two share a birthday is at least $1/2$?  This is the `Birthday Problem', and the answer is that $23$ people is enough (assuming that the birthdays are independently and identically distributed). 
 
Our starting point is an elementary inequality, which we will call the \textit{birthday inequality}\footnote{Not to be confused with the `birthday inequality' of \cite{clevenson1991majorization} that refers to the fact that the uniform distribution on birthdays minimizes the probability of a collision.}:

\begin{prop}[The Birthday Inequality]
\label{lem:BirthdayInequality}
Suppose $n$ people have birthdays each chosen independently and uniformly at random from  $m$ possible birthdays.  Let $E_n$ be the event that no two people share a birthday, and $p = 1/m$ the probability that two given people share a birthday.  Then
\[ \Pr[E_n] \le (1-p)^{\binom n 2} \, . \]
\end{prop}

\begin{proof}
Let $E_k$ be the event that there are no shared birthdays among the first $k$ people.  Let $V_k$ be  the fraction of birthdays covered by the first $k$ people.  Then
\begin{align*}
\Pr[ E_n] &= \E[1- V_{n-1} | E_{n-1}]\cdot \Pr[E_{n-1}] \,.
\end{align*}
We assume inductively that $\Pr[E_k ] \le (1- p)^{\binom k 2}$, and note that 
\begin{align*}
\E[1-V_{k} | E_k ] = 1- \frac{k}{m} \le \left (1- \frac{1}{m} \right)^k =(1- p)^k 
\end{align*}
which shows that $\Pr[E_{k+1}] \le (1-p)^{\binom{k+1}{2}}$ for all $k \ge 1$.
\end{proof}

We are interested in geometric birthday inequalities, and in particular in settings relevant to two models from statistical physics: the hard sphere model and the hard-core lattice gas model.   In these models,  particles are placed  at random in a metric space $\mathcal X$ equipped with a probability measure $\mu$ (e.g. the unit cube or a subset of the $d$-dimensional integer lattice with uniform measure) conditioned on all pairwise distances between particles being larger than some threshold $r$.  In this setting the birthday inequality supposes an upper bound on the probability that no two particles are within distance $r$ when $n$ particles are placed independently at random according to $\mu$.  

\begin{defn}
\label{def:geomIneq}
Let $X_1, X_2, \dots X_n$ be independently sampled points from a space $\mathcal X$ according to the distribution $\mu$.  Then a birthday inequality holds if
\begin{equation}
\Pr[E_n] \le (1-p)^{\binom n 2}
\end{equation}
where $E_n$ is the event $\{  \wedge _{1 \le i < j \le n} d(X_i ,X_j) > r \}$ and $p:= \Pr[ d(X_1, X_2) \le r]$.  
\end{defn}

The quantity on the right is what the probability on the left would be if all pairwise interactions were independent, and so the birthday inequality is a statement about correlations of these events.  

A related inequality is the following \textit{repulsion inequality}.

\begin{defn}
\label{defn:repulse}
In the setting above, the repulsion inequality holds if
\begin{equation}
\label{eq:repulsion}
\E [ V_k | E_k] \ge \E [V_k]
\end{equation}
where $E_k$ is the event that the centers $X_1, X_2, \dots X_k $ are at pairwise distance greater than $r$, $V_k$ is the volume fraction of $\mathcal X$ covered by the union of the closed balls of radius $r$ around $X_1, \dots X_k$, and  the expectations are taken over choosing $X_1, \dots X_k$ independently at random in $\mathcal X$ according to $\mu$. 
\end{defn}

 The repulsion inequality states that conditioning on the event that the centers of randomly placed balls of radius $r$ are at pairwise distance greater than $r$ does not decrease the expected volume of their union (as compared to the unconditional expectation).  The repulsion inequality has the flavor of a probabilistic version of the Kneser-Poulsen conjecture \cite{poulsen1954problem,kneser1955einige}: moving a set of spheres in Euclidean space so that all pairwise distances between their centers do not decrease cannot decrease the volume of the union of the spheres.  This was proved in two dimensions by Bezdek and Connelly \cite{bezdek2003pushing}, but is open in higher dimensions.
 
 While the repulsion inequality seems intuitively obvious, we show in Corollary \ref{cor:sphereRepulse} that it is not always true; in particular it fails in dimension $24$. 

As in the proof of Proposition \ref{lem:BirthdayInequality}, the birthday inequality on $n$ points is implied  if the repulsion inequality holds for all $k$  between $1$ and $n-1$.  We  write $\Pr[E_n] = (1- \E[V_{n-1}| E_{n-1}]) \cdot \Pr[E_{n-1}]$ and continue inductively. The unconditional expectation satisfies $\E[V_k] = 1- (1-p)^k$, and so if the repulsion inequality holds for all $1 \le k \le n-1$,  the birthday inequality holds.

We will show that at sufficiently low particle densities, the repulsion inequality holds in both the hard sphere and hard-core models.  This leads to bounds on the free energy in both models via the birthday inequality.  However, we will also show that at sufficiently high densities the birthday inequality can fail.  We conclude by conjecturing that the failure of the repulsion inequality can be used to indicate the fluid/solid phase transition.

\section{Hard spheres}
\label{sec:spheres}

The hard sphere model is a model of particles as randomly positioned non-overlapping spheres: a random sphere packing. There are no forces in the model besides the hard constraint that two spheres cannot overlap.  We define the hard sphere model on $\tor$, the $d$-dimensional unit torus.
\begin{defn}
\label{def:hardsphere}
The hard sphere model $H_d(n,r)$ consists of a uniformly random configuration of $n$ spheres of radius $r/2$ in $\tor$, conditioned on the event that the centers of the $n$ spheres are at pairwise distance greater than $r$.
\end{defn}

An important quantity in statistical physics is the partition function:

\begin{defn}
\label{def:SpherePartition}
The partition function, $Z_d(n,r)$, of the hard sphere model on $\tor$ is defined as:
\begin{equation}
Z_d(n,r) = \int_{\tor} \cdots \int_{\tor} \mathbf 1_{E_n} \, d x_1  \cdots d x_n 
\end{equation}
where $E_n$ is the event that $d(x_i, x_j) >r$ for all $1 \le i < j \le n$. 
\end{defn}

We define the density $\alpha$ of $H_d(n,r)$ as the fraction of volume of $\tor$ occupied by the spheres of radius $r/2$ around the $n$ centers, i.e., $\alpha = n (r/2)^d v_d$ where $v_d$ is the volume of the unit ball in $\R^d$.  As $\alpha$ is the density of the random sphere packing given by $H_d(n,r)$, it must lie between $0$ and the maximum sphere packing density in $d$ dimensions.

\begin{defn}
\label{def:SphereFE}
The free energy of the hard sphere model at density $\alpha$ is:
\begin{equation}
F_d(\alpha) = - \lim_{n \to \infty} \frac{1}{n} \log Z_d(n,r_n(\alpha))  
\end{equation}
where $r_n(\alpha) = 2(\alpha/(n v_d))^{1/d}$.  
\end{defn}

Physicists believe that the hard sphere model in dimension $d \ge 2$ undergoes a fluid / solid phase transition as the density of the spheres increases: at low densities configurations show no long-range order, while after the phase transition long-range order emerges.  For an introduction to the hard sphere model see \cite{lowen2000fun} and the references therein.  In dimension $d=1$ there is no phase transition and the model is solved; that is, an explicit expression for the free energy is known, and it has no non-analytic points \cite{tonks1936complete}.  Mathematicians have proved rigorous lower bounds on the density at which a Markov chain to sample from the model mixes rapidly \cite{kannan2003rapid, randall2005mixing, diaconis2011geometric, hayes2014lower}. See \cite{radin1987low, bowen2006fluid} for a discussion of mathematical proofs of phase transitions in continuous hard-core models, and in the second a proof of a phase transition in a system with zipper-like molecules.

We define the model with spheres of radius $r/2$ because it will be convenient to view the hard sphere model from the perspective of the random geometric graph $G_d(n,r)$: $n$ points placed uniformly and independently at random in $\tor$ with an edge placed between pairs of points at distance at most $r$.

The following proposition relates the hard sphere model to the random geometric graph, and follows immediately from Definition \ref{def:hardsphere}.
\begin{prop}
\label{prop:partitionProp}
\[ Z_d(n,r) = \Pr[G_d(n,r) \text { is empty}] .\]
\end{prop}

In what follows we parameterize both the hard sphere model and the random geometric graph by $p:= v_d r^d$, the probability that two uniformly random points in $\tor$ are at distance at most $r$.  Abusing notation we will write $G_d(n,p)$ for $G_d(n,r(p))$ where $v_d r(p)^d = p$. This parameterization gives some intuition for the birthday inequality: if $G(n,p)$ is the \ER random graph on $n$ vertices (every edge present independently with probability $p$), then the birthday inequality is $\Pr[G_d(n,p) \text{ is empty}] \le \Pr[ G(n,p) \text{ is empty}]$.  In fact, if we fix $n$ and $p$ and let $d \to \infty$, then  $\Pr[G_d(n,p) \text{ is empty}] \to \Pr[ G(n,p) \text{ is empty}]$ (Theorem 2 in \cite{devroye2011high} for the random geometric graph defined on the surface of the $d$-dimensional unit sphere). 

The birthday inequality holds in dimension $1$ for all values of $p$:

\begin{prop}
\label{sphered1prop}
For all $p \in [0,1]$, 
\[ \Pr[G_1(n,p) \text{ is empty}] \le \left( 1- p \right )^{\binom n 2} \,. \]
\end{prop}

\begin{proof}
For $p > 2/ n$, we are beyond the maximum packing density on the circle, and so $\Pr[G_1(n,p) \text{ is empty}]  =0$ and the inequality holds.  For $p \le 2/n $, we can write the left-hand side explicitly: $ \Pr[G_1(n,p) \text{ is empty}] = (1 - np/2)^{n-1}$; then it is a calculus exercise to show that $(1 - np/2)^{n-1} \le (1-p) ^{\binom n 2}$ for $p \le 2/n$.  
\end{proof}

Our first main result of this section is to show that in any dimension, at a low enough density the birthday inequality holds. We do this via the repulsion inequality (\ref{eq:repulsion}).  

\begin{theorem}
\label{thm:sphere}
For the the hard sphere model on $\tor$, for densities $\alpha \le 2^{-2-3d}$ the repulsion inequality holds.  
\end{theorem}

Theorem \ref{thm:sphere} and the birthday inequality immediately imply a lower bound on the free energy of the hard sphere model at sufficiently low densities, which to the best of our knowledge is new. 
\begin{cor}
\label{cor:SphereFreeEnergy}
For $\alpha \le 2^{-2-3d}$, $ F_d(\alpha) \ge 2^{d-1} \alpha   $. 
\end{cor}

\begin{proof}[Proof of Theorem \ref{thm:sphere}]
We first define some notation.  For a collection of $k$ centers in $\tor$, let $V_k$ be the volume of points in $\tor$ at distance at most $r$ from one of the $k$ centers, i.e. the volume of the union of balls of radius $r$ around the centers.  Let $E_k$ be the event that the $k$ centers are at pairwise distance greater than $r$. As always, we have $p = v_d r^d$, and we assume $\alpha \le 2^{-2-3d}$, i.e. $p \le 4^{-1-d}/n$.   Our goal is to prove the repulsion inequality $ \E[V_k | E_{k}] \ge \E[V_k]$ where the randomness is in placing each center uniformly and independently at random in $\tor$. 

We will prove the following estimate for all $1 \le k \le n-1$:

\begin{equation}
\label{eq:sphereEst}
\E[V_k|E_k] \ge kp - \binom {k}{2}p^2 \frac{ 1 -  4^{-d}}{(1-  kp)^2 }  \, .
\end{equation}

To complete the proof of the Theorem from (\ref{eq:sphereEst}) we use inclusion/exclusion to bound $\E [V_k] = 1-(1-p)^k \le kp - \binom{k}{2} p^2 +\binom {k}{3} p^3$, and so 
\begin{align*}
\E[V_k | E_{k}]  - \E[V_k] & \ge \binom{k}{2} p^2 \left( 1- \frac{ 1 -  4^{-d}}{(1-  kp)^2 }  - \frac{k-2}{3} p \right ) 
\end{align*}
 which is non-negative when $p \le 4^{-1-d}/n$. 

To prove (\ref{eq:sphereEst}) we use inclusion/exclusion and linearity of expectation to get the lower bound
\begin{align*}
\E[V_k | E_k] &\ge kp - \sum_{i <j} \E [V(i,j) |E_k] = kp - \binom{k}{2} \E [V(1,2) |E_k] \, ,
\end{align*}
where $V(i,j)$ is the volume covered by the balls of radius $r$ around both centers $i$ and $j$, i.e. their overlap volume.

Now let $x$ be a fixed point in $\tor$ (say the origin), $A_x^{1,2}$ the event that $x$ is covered by the balls of radius $r$ around both centers $1$ and $2$, and $E_2$ the event that centers $1$ and $2$ are at distance greater $r$.  Then we have
\begin{align*}
\E [V(1,2) |E_k] &= \Pr[A_x^{1,2} | E_k] \\
&= \frac{ \Pr[A_x^{1,2} \cap E_k] }{ \Pr[E_k]  } \\
&= \frac{ \Pr[ A_x^{1,2} \cap E_2] \cdot \Pr[E_k |A_x^{1,2} \cap E_2]}{\Pr[E_2] \cdot\Pr[E_k | E_2]} \\
&= \Pr[ A_x^{1,2} | E_2]  \cdot \frac{\Pr[E_k | A_x^{1,2}\cap E_2]}{\Pr[E_k | E_2]}  \, .
\end{align*}

First note that $\Pr[E_k | A_x^{1,2}\cap E_2] \le \Pr[E_{k-2}]$.  Next, write 
\begin{align*}
\Pr[E_k| E_2] &= \frac{ \Pr[E_{k-2}] \Pr[E_k| E_{k-2}]}{\Pr[E_2]}  \\
&\ge \Pr[E_{k-2}]  \frac{ (1-(k-2)p)(1-(k-1)p)}{1-p}\\
& \ge \Pr[E_{k-2}] \frac{(1-kp)^2}{1-p} \, ,
\end{align*}
where we have used the inequalities $\Pr[E_{k-1}|E_{k-2}] \ge 1 -(k-2)p$ and $\Pr[E_{k}|E_{k-1}] \ge 1 -(k-1)p$ which follow from the union bound: the volume of the union of balls of radius $r$ around $k-2$ centers is at most $(k-2)p$. 

This gives
\begin{align*}
\frac{\Pr[E_k | A_x^{1,2}\cap E_2]}{\Pr[E_k | E_2]} & \le \frac{1-p}{(1-kp)^2}  \, .
\end{align*}

Finally we upper bound $\Pr[A_x^{1,2} |E_2]$. We write
\begin{align*}
p^2 =\Pr[A_x^{1,2}]  &= p \Pr[A_x^{1,2}| \overline{E_2}] + (1-p) \Pr[A_x^{1,2} |E_2]  \, .
\end{align*}
The probability that three given points form a triangle in the random geometric graph with connection radius $r$ is $p \cdot \Pr[A_x^{1,2}| \overline{E_2}]$. A lower bound for the probability of forming a triangle is the probability that the first two points fall in a ball of radius $r/2$ around the third, which has probability $p^2 4^{-d}$.  Putting this together we have
\begin{align*}
 \Pr[A_x^{1,2} |E_2] &\le \frac{p^2 - p^2 4^{-d}}{1-p}
\end{align*}
and
\begin{align*}
\E [V(1,2) |E_k] &\le p^2 \frac{1 -  4^{-d}}{ (1-  kp)^2 }\, ,
\end{align*}
which gives (\ref{eq:sphereEst}).

\end{proof}

Our next result is that the birthday inequality does not hold in general.  We show this in dimension $24$, using the fact that there is a sphere packing of particularly high density.

\begin{theorem}
\label{prop:sphere}
In dimension $24$, the birthday inequality fails for large enough $n$ at densities $\alpha \in ( (.79)^{24} \cdot \rho,\rho)$, where $\rho = .001929$.
\end{theorem}

This theorem implies that the repulsion inequality fails at some density in dimension $24$:

\begin{cor}
\label{cor:sphereRepulse}
For large enough $n$, there exists some $r$  so that when $n$ spheres with centers $x_1, \dots x_n$ are placed uniformly at random in $\mathcal T^{24}$, 
\[ \E \left[ \text{vol}\left( \cup_{i=1}^n B(x_i, r)  \right)  | d(x_i, x_j) >r \text{ for all } i \ne j   \right ]  <  \E \left [ \text{vol}\left( \cup_{i=1}^n B(x_i, r)  \right)  \right  ] ,  \]
where $B(x_i, r)$ is the closed ball of radius $r$ around the center $x_i$.
\end{cor}

In other words, conditioning on the pairwise repulsion of the centers of the spheres can \textit{decrease} the expected volume of their union!  

Note that working in the torus is not essential to the result: the same holds if the centers of the spheres are chosen at random in a box in $\R^{24}$ large enough so that boundary effects are negligible.

\begin{proof}[Proof of Theorem \ref{prop:sphere}]
Consider some packing of $n$ spheres of radius $r_p$ in $\tor$ with density $\rho$.  Place a sphere of radius $(1-t) r_p$ for $0<t<1$ around each center of the packing.  If we place a new set of centers, one in each of these spheres of radius $(1-t) r_p$, then the spheres of radius $tr_p=:r/2$ around them will be disjoint.  The density of such a configuration is $\alpha = n v_d (t r_p)^d = t^d \rho$, or in other words, $t = (\alpha/\rho)^{1/d}$.  We can lower bound the probability that $n$ random centers of spheres of radius $r/2$ will be disjoint by the probability that each of the $n$ centers falls into a distinct sphere of radius $(1-t) r_p$ around the centers of the packing:
\begin{align*}
\Pr[ G_d(n, r) \text{ is empty}] &\ge \frac{n!}{n^n} ((1-t)^d \rho)^n = \frac{n!}{n^n} \left( 1- (\alpha/\rho)^{1/d} \right) ^{dn} \rho^n
\end{align*}
 and
 \begin{align*}
-\frac{1}{n} \log \Pr[ G_d(n, r) \text{ is empty}] &\le 1 - d \log  \left( 1- (\alpha/\rho)^{1/d} \right)-\log \rho + o(1) \, .
\end{align*}
The birthday inequality, however, asserts that 
\begin{align*}
-\frac{1}{n} \log \Pr[ G_d(n, r) \text{ is empty}] &\ge 2^{d-1} \alpha \, . %= 2^{d-1} (1-t)^d \rho 
\end{align*}
For $d=24$, there is a sphere packing of $\R^{24}$ of density $\frac{\pi^{12}}{12!}$ via the Leech lattice \cite{leech1967notes, cohn2004densest}. For any $\eps >0$, and all large enough $n$, we can find a packing of $\mathcal T^{24}$ with $n$ non-overlapping spheres which has density at least $\frac{\pi^{12}}{12!} - \eps$.  Choosing $\rho = .001929 \approx  \frac{\pi^{12}}{12!} - 5\cdot10^{-7}$, we can compare the birthday inequality lower bound on the free energy $BI(\rho,t)$ to the cell model upper bound $CM(\rho,t)$. 
\begin{align*}
F(t) := BI(\rho,t) - CM(\rho,t) &= \frac{\rho}{2} (2t)^{24} - 1 + 24\log(1- t) + \log(\rho)  \, .
\end{align*}
A calculation gives $F(.79) >0$ and $F^\prime(t) >0$ for $t\in (.79,1)$. This proves Theorem \ref{prop:sphere}.  Corollary \ref{cor:sphereRepulse} follows immediately since the sequence of repulsion inequalities implies the birthday inequality. 
\end{proof}

The \textit{hard square model} is the hard sphere model under $l_\infty$ distance: configurations of disjoint $d$-dimensional axis-parallel cubes.  Cubes pack particularly nicely, with a maximum packing density of $1$.  We use this to show that the birthday inequality fails in all sufficiently high dimensions.

\begin{theorem}
\label{thm:square}
The birthday inequality fails in the hard square model for some range of densities in dimension $d \ge 6$. For $d=6$, the inequality fails for $\alpha \in (.4 , .95)$.  For $d > 6$, the inequality fails for $\alpha \in (\underline \eta_d, \overline \eta_d)$, where   $\underline \eta_d \sim 2^{1-d} \log (2) \cdot d$ as $d \to \infty$,   $\overline \eta_d \to  1$ as $d \to \infty$.   
\end{theorem}

\begin{proof}
Here we ask for a given $d$ if there is some $\alpha \in (0,1)$ so that
\[ 1- d \log(1- \alpha^{1/d}) > 2^{d-1} \alpha \, . \]
A numerical calculation for $d=6$ and some calculus give the theorem. 
\end{proof}

\section{The hard-core model}
\label{sec:hardcore}

Assume $n$ is such that $n^{1/d}$ is an even integer. Let $\Z_d(n)$ be the $d$-dimensional discrete torus of sidelength $n^{1/d}$ (with a total of $n$ sites).  Assume $\alpha$ is such that $\alpha n$ is an integer.  We define a fixed-density hard-core model as follows:

\begin{defn}
\label{def:hardcore}
The fixed-density hard-core model $HC_d(n,\alpha)$ for $\alpha \in [0, 1/2]$ consists of a uniformly chosen random independent set of  $k= \alpha n $ sites in $\Z_d(n)$.
\end{defn}

This model is a natural discretization of the hard sphere model. It is closely related to the hard-core model with an activity parameter $\lambda$: an independent set  $I \subset \Z_d(n)$  chosen with probability proportional to $\lambda^{|I|}$.  By conditioning on $|I| =  \alpha n $ we obtain the fixed-density hard-core model defined above. In the terminology of statistical physics, the fixed-density model is the canonical ensemble, while the activity parameter model is the grand canonical ensemble.  

We can define the partition function and free energy of the hard-core model:

\begin{defn}
\label{def:LatticePartition}
The partition function, $Z_d(n,k)$, of the hard-core model is defined as:
\begin{equation}
Z_d(n,k) = \text{IS(k)} :=  \# \text{ of independent sets of size k in } \Z_d(n)  \, .
\end{equation}
\end{defn}

We can write $Z_d(n,k) = \frac{n^k}{k!} \Pr[ X_k \text{ is an independent set}]$ where $X_k$ is a (multi)-set of $k$ independent and uniformly chosen sites from $\Z_d(n)$. 

\begin{defn}
\label{def:LatticeFE}
The free energy of the hard-core model at density $\alpha$ is:
\begin{equation}
F_d(\alpha) =  \lim_{n \to \infty} \frac{1}{n} \log Z_d(n,\alpha n)   \, .
\end{equation} 
\end{defn}
Note that we do not take the negative of the log partition function here, and so while we obtained lower bounds on the free energy of the hard sphere model, here we will obtain upper bounds. 

We can write the free energy in terms of the probability that $X_k$ is an independent set:
\begin{equation}
F_d(\alpha) = \alpha - \alpha \log \alpha  + \lim _{n \to \infty} \frac{1}{n} \log \Pr [ X_{\alpha n} \text{ is an independent set} ] \,.
\end{equation}

We can also define the fixed-size hard-core model on the $d$-dimensional Hamming cube $Q_d$, with vertex set $\{0,1\}^d$ and edges between vectors that differ in exactly one coordinate. The partition function and free energy are defined as in the hard-core model on $\Z_d(n)$.

Our first theorem of this section is that  the repulsion inequality (and thus the birthday inequality) holds in the hard-core model at a sufficiently low density on any $d$-regular graph.  We consider a $d$-regular graph $G$ on $n$ vertices, and select a set of $k$ vertices $X_k$ uniformly at random with replacement. We use the convention that two vertices in $X_k$ form an edge if they are neighbors in $G$ or if they are identical; so if $X_k$ has no edges,  it is an independent set of size $k$ in  $G$.  Thus we have $p = \frac{d+1}{n}$, the probability that two randomly chosen vertices form an edge.

\begin{theorem}
\label{thm:bipartite}
For the the hard-core model on any $d$-regular graph $G$ on $n$ vertices, at densities $\alpha \le (d+1)^{-2}$, the repulsion inequality (\ref{eq:repulsion}) holds.  
\end{theorem}

As a corollary via the birthday inequality we get an improved upper bound on the number of independent sets of size $\alpha n$ in all $d$-regular graphs, for $\alpha \le (d+1)^{-2}$.

\begin{cor}
\label{cor:bipartite}
For $\alpha \le (d+1)^{-2}$, the number of independent sets of size $\alpha n$ in any $d$-regular graph $G$ satisfies:
\[  \text{IS}(\alpha n) \le \frac{n^{\alpha n}}{(\alpha n)! } \left ( 1- \frac{d+1}{n} \right) ^{\binom {\alpha n}{2}} \, . \]
On the scale of the free energy, this gives:
\begin{equation}
\label{eq:indsetbound}
 \frac{1}{n} \log \text{IS}(\alpha n) \le \alpha  -  \alpha \log \alpha - \alpha^2 \frac{d+1}{2} .  
 \end{equation}
\end{cor}

For $\alpha \le (d+1)^{-2}$, Corollary \ref{cor:bipartite} improves the bound given by Carroll, Galvin, and Tetali in \cite{carroll2009matchings}\footnote{For this range of $\alpha$, the best upper bound in \cite{carroll2009matchings} on $IS(\alpha n)$ is the third bound given in Theorem 1.6, $2^{\alpha n} \binom{n/2}{\alpha n}$. On the scale of the free energy this is $ - \alpha \log( \alpha) - (1/2 -\alpha) \log (1- 2 \alpha) $. Some calculus shows that the bound \eqref{eq:indsetbound} is lower for $\alpha \le (d+1)^{-2}$.}.    Specializing to $\Z_d(n)$ and $Q_d$ we get upper bounds of \\$\alpha \left ( 1 -  \log \alpha - \alpha (2d+1)/2 \right )$ and $ \alpha \left ( 1 -  \log \alpha - \alpha (d+1)/2 \right )$ respectively on the normalized logarithm of the number of independent sets of size $\alpha n$. As far as we know these are the best bounds known on the number of independent sets of a given size in $\Z_d(n)$ and $Q_d$ at these densities.

\begin{proof}[Proof of Theorem \ref{thm:bipartite}]
The proof is essentially the same as the proof of Theorem \ref{thm:sphere}, and this is one of the motivations of this work: to find methods for analyzing the hard-core model that generalize to the hard sphere model. 

Let $V_k$ be the fraction of vertices in $G$ at distance at most $1$ to a set of $k$ randomly chosen vertices.  Let $E_k$ be the event that the set of $k$ random vertices is at pairwise distance at least $2$ in $G$. We can assume that $k \ge 2$ (and thus $n \ge 2 (d+1)^2$) since the case $k=1$ is immediate.  Let $p = \frac{d+1}{n}$, the probability that two randomly chosen vertices in $G$ coincide or are neighbors.   We will prove the following estimate:
\begin{equation}
\label{eq:BipEst}
\E[V_k|E_k] \ge \frac{k(d+1)}{n} - \binom {k}{2} \frac{d(d-1)}{ n^2 (1-kp)^2} 
\end{equation}
where the randomness is in selecting the $k$  vertices uniformly and independently at random. 

 By inclusion/exclusion we have 
\[ \E [V_k] = 1-(1-p)^k \le kp - \binom{k}{2} p^2 +\binom {k}{3} p^3 = \frac{k(d+1)}{n} - \binom{k}{2} \frac{(d+1)^2}{n^2} \left( 1 - \frac{(k-2)(d+1)}{3n} \right ) \]
 and using (\ref{eq:BipEst}) we get
\begin{align*}
\E[V_k|E_k] - \E[V_k] &\ge  \binom{k}{2}p^2 \left(1 - \frac{d(d-1)}{(d+1)^2}\cdot \frac{1}{ (1-kp)^2} - \frac{(k-2)(d+1)}{3n}    \right)   \\
&\ge \binom{k}{2}p^2 \left(1 - \frac{d(d-1)}{(d+1)^2}\cdot \frac{1}{ (1-kp)^2} - \frac{kp}{3}    \right)  \, .
\end{align*}
This is non-negative when $\alpha = k/n \le \frac{1}{(d+1)^2}$: the RHS is decreasing in $k$, and so it is enough to prove when $k = n/(d+1)^2$. This follows from an elementary calculation and proves Theorem \ref{thm:bipartite}, modulo the estimate (\ref{eq:BipEst}).  

To prove (\ref{eq:BipEst}), we use inclusion/exclusion again to bound 
\[  \E[V_k|E_k] \ge kp - \binom {k}{2} \E[V(1,2)|E_k], \]
where $V(1,2)$ is the fraction of vertices in $G$ at distance $0$ or $1$ of the first and second of the $k$ randomly selected vertices.  Let $A_v^{1,2}$ be the event that vertex $v$ neighbors both of the first two selected vertices.  We write
 \begin{align*}
\E[V(1,2)| E_k] &= \frac{ \E[ V(1,2) \cdot \mathbf 1_{E_k}] }{\Pr[E_k]} \\ 
&= \frac{1}{n} \sum_{v \in G} \frac{\Pr[A_v^{1,2} \cap E_k]}{\Pr[E_k]} \\
&= \frac{1}{n} \sum_{v \in G} \frac{ \Pr[ A_v^{1,2} \cap E_2] \cdot \Pr[E_k | A_v^{1,2} \cap E_2]}{\Pr[E_2] \cdot\Pr[E_k | E_2]} \\
&= \frac{1}{n} \sum_{v \in G}\Pr[ A_v^{1,2} | E_2]  \cdot \frac{ \Pr[E_k |A_v^{1,2} \cap E_2]}{\Pr[E_k | E_2]} \, .
\end{align*}
If the neighbors of $v$ form a clique then that term in the sum is $0$.  We assume from here that there is at least one edge missing from the subgraph of $v$'s neighbors.   

Consider $ \frac{  \Pr[E_k |A_v^{1,2} \cap E_2] }{\Pr[E_k | E_2]}$. Again we have $\Pr[E_k |A_v^{1,2} \cap E_2] \le \Pr[E_{k-2}]$, and
\begin{align*}
\Pr[E_k | E_2] & \ge \frac{\Pr[E_{k-2}]\Pr[E_k | E_{k-2}]   }{\Pr[E_2]   }  \\
&\ge  \frac{\Pr[E_{k-2}] (1-kp)^2      }{ 1-p  }
\end{align*}
which gives
\begin{equation}
\label{eq:union}
\frac{ \Pr[E_k | A_v^{1,2} \cap E_2]}{\Pr[E_k | E_2]} \le \frac{1-p }{(1-kp)^2}  \, .
\end{equation}
Next note that for any $d$-regular graph $G$, 
\begin{equation}
\label{eq:triangles}
\frac{1}{n} \sum_{v \in G}  \Pr[ A_v^{1,2} | E_2] = \frac{1}{n} \cdot \frac{n \binom d 2 - 3 \cdot \#\{ C_3's \text{ in } G  \}  }{\binom n 2 - dn/2   } \le \frac{d}{n} \cdot \frac{d-1}{n(1-p)} \, .
\end{equation}
Inequalities (\ref{eq:union}) and (\ref{eq:triangles}) give 
 \begin{align*}
\E[V(1,2)| E_k] &\le    \frac{d(d-1)}{n^2 (1-kp)^2}  
\end{align*}
 and thus  (\ref{eq:BipEst}).

\end{proof}

We now show that the birthday inequality fails in general for $d$-regular, bipartite graphs with $d \ge 6$.  

\begin{theorem}
\label{prop:BipFail}
For $d \ge 6$, there exist constants $\alpha_d^l \in (0, 1/2)$ so that for $n$ large enough, the birthday inequality fails for the hard-core model on any $d$-regular, bipartite graph $G$ on $n$ vertices at densities $\alpha \in [ \alpha_d^l, 1/2]$.  Asymptotically, $\alpha_d^l \sim 2 \log 2 /d$ as $d \to \infty$.    
\end{theorem}

\begin{proof}

For a lower bound on the number of independent sets of size $\alpha n$ in $G$, we use the \textit{parity} lower bound: any subset of one side of the bipartition is an independent set, and so 
\begin{align*}
\text{IS}(\alpha n) &\ge \binom{n/2}{\alpha n}  \, .
\end{align*}
The corresponding bound on the free energy is 
\begin{align*}
F_d (\alpha) &\ge - \alpha \log (2 \alpha)  -(1/2 -\alpha) \log(1- 2 \alpha) +o(1) \, .
\end{align*}
The birthday inequality  asserts the upper bound:
\begin{align*}
F_d(\alpha) & \le  \alpha \left ( 1 -  \log \alpha - \alpha \frac{d+1}{2} \right ) +o(1) \, .
\end{align*}

Some calculus shows that these bounds cross for $d\ge 6$, and that asymptotically as $d \to \infty$, the crossing point is  $\alpha_d^l \sim 2 \log 2 /d$.

\end{proof}

\section{Matchings}
\label{sec:matchings}

In this section we use the repulsion inequality to give bounds on the number of matchings of size $k$ in a $d$-regular graph $G$ on $n$ vertices.  Such a graph has $nd/2$ edges, and each edge shares a vertex with $2d-2$ other edges.  We let $p = \frac{2d-1}{nd/2}$, the probability that two uniformly chosen random edges (with replacement) coincide or intersect at a vertex.  Then the birthday inequality asserts that $\Pr[E_k] \le (1-p)^{\binom k 2}$, where $E_k$ is the event that  $k$ edges  chosen uniformly at random from $G$ form a matching of size $k$.  The repulsion inequality states that $\E[V_k | E_k] \ge \E[V_k]$, where $V_k$ is the fraction of edges covered by a set of $k$ edges: the fraction of edges that are contained in or intersect the set.   Since a matching in $G$ is an independent set in the line graph of $G$, we could apply Theorem \ref{thm:bipartite} to get a bound, but we can do better working directly, since for $d\ge 3$, the line graph of a $d$-regular graph contains many triangles. Let $M(k)$ be the number of matchings consisting of $k$ edges in $G$. We show:

\begin{theorem}
\label{thm:matchings}
For $\alpha \le \frac{3}{28} $, the repulsion inequality holds for matchings of size $\alpha \frac{n}{2}$ in a $d$-regular graph on $n$ vertices, and as a consequence
\[ M(\alpha n/2) \le \frac{(nd/2)^{\alpha n/2}}{(\alpha n/2)!} \left(1- \frac{2d-1}{nd/2}    \right )^{\binom {\alpha n /2}{2}}  .   \]
On a logarithmic scale, this gives
 \begin{equation}
\label{eq:logmatching}
 \frac{2}{n} \log M(\alpha n /2)  \le \alpha \log d -\alpha \log \alpha + \alpha  -\frac{\alpha^2}{2} \frac{2d-1}{d} . 
 \end{equation}
\end{theorem}

For $\alpha = O (d^{-1/3}) $,  Theorem \ref{thm:matchings} improves the bound 
 given by Ilinca and Kahn in \cite{ilinca2013asymptotics}\footnote{The bound in \cite{ilinca2013asymptotics}, translated to natural logarithms, is \\ $\frac{2}{n} \log M(\alpha n /2)  \le \alpha \log d - \alpha \log \alpha -2(1-\alpha)\log(1-\alpha)  -\alpha +(\log d)/(d-1)$. Subtracting the first two matching terms then power expanding around $\alpha =0$ gives $ \alpha - \alpha^2 + (\log d)/(d-1)- \alpha^3/3 - \dots $ for the Ilinca-Kahn bound and $\alpha -\alpha^2 +\alpha^2/(2d) $ for the birthday inequality bound  (\ref{eq:logmatching}). These cross when $\alpha =\Theta(((\log d)/d)^{1/3})$.  In particular, for all larger $\alpha$, the Ilinca-Kahn bound is stronger than the birthday inequality, giving Corollary \ref{cor:BImatch}.}.  Together Theorem \ref{thm:matchings} and \cite{ilinca2013asymptotics} show that the birthday inequality holds for matchings of all sizes in $d$-regular graphs.

\begin{cor}
\label{cor:BImatch}
The birthday inequality holds for matchings of size $k$, for all $k$, in every $d$-regular graph on $n$ vertices.
\end{cor}

In would be nice to prove that in fact the repulsion inequality holds for matchings of any size in a $d$-regular graph.

\begin{proof}[Proof of Theorem \ref{thm:matchings}]
Let $m = nd/2$ be the number of edges of $G$, and $p = \frac{2d-1}{m}$. We want to show that $\E[V_k | E_k ]\ge \E[V_k]$.  Let $L_2$ be the number of edges of $G$ that are covered by two edges of a matching of size $k$ but are not part of the matching themselves.  Then
\begin{align*}
\E [ V_k | E_k] &= k+ (2d-2)k - \E [L_2 | E_k]  = m \left(  p k - \E [L_2/m| E_k] \right) \, .
\end{align*}
By inclusion/exclusion we have
\begin{align*}
\E[V_k] &= m(1- (1-p)^k) \le  m pk - m p^2 \binom {k}{2} + m p^3 \binom {k}{3} \\
&\le m \left( pk - \binom {k}{2}p^2 \left(1 - \frac{kp}{3}   \right)  \right ) \, .
\end{align*}
So it is enough to show that
\begin{align*}
\E[L_2 | E_k] &\le  m \binom {k}{2}p^2 \left(1 - \frac{kp}{3}   \right) \, .
\end{align*}
We can write
\begin{align*}
\E[L_2 | E_k] &= \sum_{e \in G} \binom{k}{2} \Pr[A_2^{1,2} | E_k]
\end{align*}
where $A_e^{1,2}$ is the event that edge $e$ is covered by edges $1$ and $2$. Now it is enough to show that $\Pr[A_2^{1,2} | E_k] \le p^2 (1-kp/3)$. As in the proofs above we write
\begin{align*}
\Pr[A_2^{1,2} | E_k] &= \Pr[ A_e^{1,2} | E_2]  \cdot \frac{ \Pr[E_k |A_e^{1,2} \cap E_2]}{\Pr[E_k | E_2]} \\
&\le \Pr[ A_e^{1,2} | E_2] \cdot \frac{1-p}{(1-kp)^2}  \, .
\end{align*}
We calculate 
\[ \Pr[ A_e^{1,2} | E_2]  = \frac{2(d-1)^2}{m^2 (1-p)}  \]
which gives
\[ \Pr[A_2^{1,2} | E_k] \le \frac{ 2 (d-1)^2}{m^2 (1-kp)^2} \,.\]
Our assumption that $\alpha \le 3/28$ implies that $kp \le \frac{3}{14} $, and so
\begin{align*}
\Pr[A_2^{1,2} | E_k] &\le \frac{ 2 (d-1)^2}{m^2 (1-kp)^2}   \le  \frac{(2d-1)^2}{m^2} \left(1 - \frac{kp}{3} \right )=p^2 (1-kp/3)
\end{align*}
which shows that the repulsion inequality holds. 
\end{proof}

\section{Conclusions and conjectures}
\label{sec:conjectures}

We conjecture that the lower bounds on the density at which the birthday inequality holds in Theorem \ref{thm:sphere} and Theorem \ref{thm:bipartite} can be extended to the entire fluid phase of the hard sphere and hard-core models.  

We describe two notions of the fluid phase of the hard sphere and fixed-size hard-core model.  The first is decay of correlations:

\begin{defn}
Let $x_0$, $x_0^\prime$, $x_t$ be positions in $\tor$ or lattice sites on $\Z_d(n)$. Let $A_0$ (resp. $A_0^\prime, A_t$) be the event that the position $x_0$ ($x_0^\prime, x_t$) is covered by a sphere in the hard sphere model or occupied by a particle in the hard-core model.  Then the model exhibits  decay of correlations at density $\alpha$ if there is some constant $c_\alpha >0$ so that 
\[ \left | \Pr[ A_{t} | A_0]   -\Pr[A_t | A_0^\prime  \right] | \le g( c_\alpha d_t/r)    \]
where $d_t = \min \{ d(x_0, x_t), d(x_0^\prime,x_t) \}$ and $g(s)$ is some function so that $\lim_{s \to \infty} g(s) =0$.  (For the hard-core model, we take $r=1$).  The model exhibits exponentially fast decay of correlations if we can take $g(s) \le e^{-c s}$ for some $c>0$. 
\end{defn}

The second notion is the rapid mixing of a specific Markov chain with the hard sphere or hard-core distribution as its stationary distribution.  One such chain is the single-particle, global-move dynamics (see e.g.\cite{hayes2014lower}). A single move of the Markov chain consists of selecting one center of a sphere or one particle on the lattice uniformly at random, then selecting a position or a site uniformly at random from $\tor$ or $\Z_d(n)$ and moving the center or the particle to the new position as long as it does not violate the hard constraints of the model.  We say the chain mixes rapidly if the mixing time is a polynomial in $n$.  

\begin{conj}
\label{conj:phase}
If the hard sphere or hard-core model is in the fluid phase (say it  exhibits exponentially fast decay of correlations or the Markov chain above mixes rapidly), then the repulsion inequality (and thus the birthday inequality) holds.
\end{conj}

The intuition behind Conjecture \ref{conj:phase} is that at a sufficiently low density, conditioning on particles being repulsed from each other should have an essentially local effect, and locally, conditioning on repulsion increases the volume covered by the union of balls around the particles.  However, beyond the fluid/solid phase transition, long-range correlations come into play, and conditioning on the repulsion of particles can force them into global lattice-like configurations with holes, and thus the volume covered may actually decrease.  Note that a model of random matchings of a given size on the $d$-dimensional lattice, the monomer-dimer model, does not exhibit a phase transition \cite{heilmann1972theory}, and Corollary \ref{cor:BImatch} shows that the birthday inequality holds at all densities.

Conjecture \ref{conj:phase} has several consequences.  First, it would give a mathematical proof that the hard sphere model in dimension $24$ undergoes a phase transition, which to the best of our knowledge has not been proved yet in any dimension.  The density at which the birthday inequality fails would mark an upper bound on the critical density for the model, as exponential decay of correlations (or fast mixing) could not hold.  

  Second, it would imply that the critical density in the fixed-size hard-core model is upper bounded by $\alpha_d^l$ from Theorem \ref{prop:BipFail}, in particular showing that a phase transition occurs at densities $O(1/d)$ on $\Z_d(n)$. The best known analogous bounds in the hard-core model with fugacity parameter $\lambda $ are $\lambda_c = \tilde O(d^{-1/3})$ given by Peled and Samotij \cite{peled2014odd} improving the bound of $\tilde O(d ^{-1/4})$ from Galvin and Kahn \cite{galvin2004phase}.  Proving Conjecture \ref{conj:phase} would give the optimal bound up to a constant factor $\lambda_c = O(d^{-1})$ as the typical particle density $\alpha$ in the hard-core model with fugacity $\lambda$ is bounded below by $\frac{\lam}{6 (1+ \lam)}$.

The next conjecture concerns independent sets in $d$-regular graphs. 

\begin{conj}
\label{conj:IS}
Suppose $2d$ divides $n$.  Let $H_{d,n}$ be the graph consisting of $n/(2d)$ disjoint copies of $K_{d,d}$, the complete bipartite graph on two sets of $d$ vertices. Then for all $1 \le k \le n$, $H_{d,n}$ minimizes the expected number of neighbors of a uniformly random independent set of size $k$ over all $d$-regular graphs on $n$ vertices. 
\end{conj}

In the notation above, Conjecture~\ref{conj:IS} states $\E_{H_{d,n}} [ V_k | E_k ] \le \E_G[V_k | E_k]$ for any $d$-regular $G$ on $n$ vertices.  Conjecture~\ref{conj:IS} immediately implies a theorem of Kahn \cite{kahn2001entropy} and Zhao \cite{zhao2010number} that $H_{d,n}$ maximizes the total number of independent sets in any $d$-regular graph on $n$ vertices, and implies the conjecture of Kahn \cite{kahn2001entropy} that $H_{d,n}$ in fact maximizes the number of independent sets of size $k$, for every $k$. Conjecture~\ref{conj:IS} is in fact much stronger: it states that $H_{d,n}$ maximizes the ratio of the number of independent sets of size $k$ to the number of independent sets of size $k-1$ for all $k$.  The elementary proof of Corollary~\ref{cor:bipartite} suggests that proving Conjecture~\ref{conj:IS} may in fact be easier than trying to bound the number of independent sets of a given size directly.  

\section*{Acknowledgments} Thanks to the Institute for Mathematics and its Applications in Minneapolis, MN, where part of this work was done, for its generous support.  Thanks to Alfredo Hubard for many fruitful discussions on this topic.  Thanks to Peter Winkler and Tyler Helmuth for their careful reading of a draft of this paper.


\bibliographystyle{amsplain}

\begin{thebibliography}{10}

\bibitem{bezdek2003pushing}
K\'aroly Bezdek and Robert Connelly.
\newblock Pushing disks apart.
\newblock {\em J. Reine Angew. Math}, 553:221--236, 2003.

\bibitem{bowen2006fluid}
Lewis Bowen, Russell Lyons, Charles Radin, and Peter Winkler.
\newblock Fluid-solid transition, in a hard-core system.
\newblock {\em Physical review letters}, 96(2):025701, 2006.

\bibitem{carroll2009matchings}
Teena Carroll, David Galvin, and Prasad Tetali.
\newblock Matchings and independent sets of a fixed size in regular graphs.
\newblock {\em Journal of Combinatorial Theory, Series A}, 116(7):1219--1227,
  2009.
  
\bibitem{clevenson1991majorization}
M~Lawrence Clevenson and William Watkins.
\newblock Majorization and the birthday inequality.
\newblock {\em Mathematics Magazine}, pages 183--188, 1991.

\bibitem{cohn2004densest}
Henry Cohn and Abhinav Kumar.
\newblock The densest lattice in twenty-four dimensions.
\newblock {\em Electronic Research Announcements of the American Mathematical
  Society}, 10(7):58--67, 2004.

\bibitem{devroye2011high}
Luc Devroye, Andr{\'a}s Gy{\"o}rgy, G{\'a}bor Lugosi, and Frederic Udina.
\newblock High-dimensional random geometric graphs and their clique number.
\newblock {\em Electronic Journal of Probability}, 16:2481--2508, 2011.

\bibitem{diaconis2011geometric}
Persi Diaconis, Gilles Lebeau, and Laurent Michel.
\newblock Geometric analysis for the metropolis algorithm on lipschitz domains.
\newblock {\em Inventiones mathematicae}, 185(2):239--281, 2011.

\bibitem{galvin2004phase}
David Galvin and Jeff Kahn.
\newblock On phase transition in the hard-core model on $\mathbf {Z} ^d$.
\newblock {\em Combinatorics, Probability and Computing}, 13(02):137--164,
  2004.

\bibitem{hayes2014lower}
Thomas~P Hayes and Cristopher Moore.
\newblock Lower bounds on the critical density in the hard disk model via
  optimized metrics.
\newblock {\em arXiv preprint arXiv:1407.1930}, 2014.

\bibitem{heilmann1972theory}
Ole~J Heilmann and Elliott~H Lieb.
\newblock Theory of monomer-dimer systems.
\newblock {\em Communications in Mathematical Physics}, 25(3):190--232, 1972.

\bibitem{ilinca2013asymptotics}
Liviu Ilinca and Jeff Kahn.
\newblock Asymptotics of the upper matching conjecture.
\newblock {\em Journal of Combinatorial Theory, Series A}, 120(5):976--983,
  2013.

\bibitem{kahn2001entropy}
Jeff Kahn.
\newblock An entropy approach to the hard-core model on bipartite graphs.
\newblock {\em Combinatorics, Probability and Computing}, 10(03):219--237,
  2001.

\bibitem{kannan2003rapid}
Ravi Kannan, Michael~W Mahoney, and Ravi Montenegro.
\newblock Rapid mixing of several markov chains for a hard-core model.
\newblock In {\em Algorithms and computation}, pages 663--675. Springer, 2003.

\bibitem{kneser1955einige}
Martin Kneser.
\newblock Einige bemerkungen {\"u}ber das minkowskische fl{\"a}chenmass.
\newblock {\em Archiv der Mathematik}, 6(5):382--390, 1955.

\bibitem{leech1967notes}
John Leech.
\newblock Notes on sphere packings.
\newblock {\em Canad. J. Math}, 19(251):267, 1967.

\bibitem{lowen2000fun}
Hartmut L{\"o}wen.
\newblock Fun with hard spheres.
\newblock In {\em Statistical physics and spatial statistics}, pages 295--331.
  Springer, 2000.

\bibitem{peled2014odd}
Ron Peled and Wojciech Samotij.
\newblock Odd cutsets and the hard-core model on $\mathbf {Z} ^d$.
\newblock {\em Annales de l'Institut Henri Poincare, Probabilites et
  Statistiques}, 50:975--998, 2014.

\bibitem{poulsen1954problem}
E~Thue Poulsen.
\newblock Problem 10.
\newblock {\em Mathematica Scandinavica}, 2:346--346, 1954.

\bibitem{radin1987low}
Charles Radin.
\newblock Low temperature and the origin of crystalline symmetry.
\newblock {\em International Journal of Modern Physics B}, 1(05n06):1157--1191,
  1987.

\bibitem{randall2005mixing}
Dana Randall and Peter Winkler.
\newblock Mixing points on a circle.
\newblock In {\em Approximation, Randomization and Combinatorial Optimization.
  Algorithms and Techniques}, pages 426--435. Springer, 2005.

\bibitem{tonks1936complete}
Lewi Tonks.
\newblock The complete equation of state of one, two and three-dimensional
  gases of hard elastic spheres.
\newblock {\em Physical Review}, 50(10):955, 1936.

\bibitem{zhao2010number}
Yufei Zhao.
\newblock The number of independent sets in a regular graph.
\newblock {\em Combinatorics, Probability and Computing}, 19(02):315--320,
  2010.

\end{thebibliography}

\end{document}